\documentclass[11pt]{amsart}
\usepackage[utf8]{inputenc}
\usepackage[T1]{fontenc}
\usepackage[british]{babel}
\usepackage{
  lmodern,    
  microtype,  
  amssymb,
  mathtools,
  paralist,
  tabularx,
  booktabs,
  changepage, 
}
\mathtoolsset{mathic}

\usepackage[all,cmtip]{xy}
\usepackage[font=small,format=plain,width=\textwidth]{caption}
\usepackage[dvipsnames]{xcolor}
\numberwithin{equation}{section}

\theoremstyle{theorem}

\newtheorem*{theorem*}{Theorem}
\newtheorem{lemma}[equation]{Lemma}
\newtheorem{proposition}[equation]{Proposition}
\newtheorem{corollary}[equation]{Corollary}

\theoremstyle{definition}

\theoremstyle{remark}
\newtheorem{remark}[equation]{Remark}
\newtheorem{example}[equation]{Example}
\newtheorem{question}[equation]{Question}

\newcommand*{\Z}{\mathbb{Z}}

\newcommand*{\R}{\mathbb{R}}
\newcommand*{\C}{\mathbb{C}}

\newcommand*{\sheaf}[1]{{\mathcal #1}}
\newcommand*{\Oo}{\sheaf{O}}

\newcommand*{\LC}{\sheaf{L}}

\newcommand*{\A}{\mathbb{A}}
\renewcommand*{\P}{\mathbb{P}}

\usepackage[draft=false]{hyperref}
\hypersetup{breaklinks=true,colorlinks=true,linkcolor=black,anchorcolor=black,citecolor=black,filecolor=black,menucolor=black,urlcolor=MidnightBlue}

\allowdisplaybreaks

\begin{document}

\title{A few computations about the real cycle class map in low dimensions}

\author{Jens Hornbostel}
\address{Jens Hornbostel,  Bergische Universit\"at Wuppertal, Gaußstraße 20, 42119 Wuppertal, Germany}
\email{hornbostel at math.uni-wuppertal.de}
\thanks{Research for this publication was conducted in the framework of the DFG Research Training Group 2240: Algebro-Geometric Methods in Algebra, Arithmetic and Topology.}

\begin{abstract}
We investigate the surjectivity of the real cycle class map from $I$-cohomology to classical intergral cohomology for some real smooth varieties, in particular surfaces. This might be considered as one of several possible incarnations of real integral Hodge theory.
\end{abstract}

\maketitle


\section{Introduction}

Given a smooth projective $d$-dimensional variety over $\C$, the Hodge conjecture predicts that the cycle class map $CH^*(X) \to H^{2*}(X(\C),\Z)$ is surjective when restricting the target $H^{2k}(X(\C),\Z)$ to elements of Hodge type $(k,k)$ and when passing to rational coefficients. With integral coefficients, one still has an epimorphism for $k=0,1$ and $d$, but in general there are non-algebraic integral cohomology classes satisfying the Hodge condition, as first discovered for $k=2$ by \cite{AH62}.

\medskip

For real varieties, there is a classical real cycle class map $CH^*(X)/2 \to H^{*}(X(\R),\Z/2)$ due to Borel and Haefliger \cite{BH61}, which has been extensively studied in real algebraic geometry. Again, one may ask if a given cohomology class is algebraic. However, according e.g. to \cite[Remark III.2.3]{Si89}, already for real surfaces, in the ''generic'' case elements in $H^1(X(\R),\Z/2)$ are not of algebraic orign. 

\medskip

There are two interesting integral lifts of this mod $2$ real cycle class map $c_{BH}$. The first involves a factorization through the $G$-equivariant integral cohomology $H^{2k}_G(X(\C),\Z(k))$ of the topological $G$-space $X(\C)$, where $G=C_2=\Z/2$. When investigating the surjectivity of this map, that is the algebraicity of elements in $H^{2k}_G(X(\C),\Z(k))$, we have the Hodge condition and another condition involving Steenrod squares, discovered by Kahn for chern classes and in general by Krasnov \cite{Kr94}. The surjectivity question for this map and elements satisfying both these conditions has been investigated in great detail in \cite{BW20}, \cite{BW20b}, and similarly to the complex case it sometimes fails for $2 \leq k \leq d-1$. One should also note that this incarnation of real Hodge theory reduces to classical complex Hodge theory when passing to rational coefficients, see \cite[Remark 2.4 (ii)]{BW20}.

\medskip

There is another lift of the Borel-Haefliger map, which from the point of view of motivic homotopy theory is an even better real analogue of the complex cycle class map, and which is due to Jacobson, building on the work of many others. Namely, still for $X$ a real smooth variety of dimension $d$, we have the {\em real cycle class map} $c_{\R}:H^*(X,I^*)\to H^*(X(\R),\Z)$. Here the $I$-cohomology is closely related to the Chow-Witt groups of Barge-Morel and Fasel \cite{BM00}, \cite{faselthesis}. We refer to \cite[Proposition 2.11]{HW19} for the precise statement concerning this relationship. After some initial results of Fasel for $*=d$, the real cycle class map $c_{\R}$ has been studied more systematically by Jacobson \cite{jacobson} and Hornbostel, Wendt, Xie and Zibrowius \cite{HWXZ}. See \cite[beginning of section 3]{HWXZ} for a short summary about different possible constructions resp. descriptions of $c_{\R}$. We will recall a few known facts about $c_{\R}$ in section \ref{recollections} below.
 
\medskip
 
The purpose of this note is to further investigate the surjectivity of the real cycle class map, notably for certain real varieties of small dimension. For instance, for some real surfaces $X$ all elements $H^1(X(\R),\Z)$ will be of algebraic orign, and for others some elements will not. These results are just a very modest start, providing essentially the first explicit computations of the real cycle class map in non-cellular examples and cohomological degree different from $d$. More subtle computations will be certainly carried out by other authors in forthcoming work.  

\medskip

Throughout this article, $X$ always denotes a smooth real variety (i.e. irreducible and reduced, of finite type) of dimension $d$, and $X(\R)$ denotes the set of real points, equipped with the analytic topology. For bigraded $I$-comology we refer to \cite{faselthesis}, and e.g. to \cite{AF16} for a sheaf-theoretic description. We will be mostly interested in groups of bidegrees $H^k(X,I^k)$, as these are the source of the real cycle class map $c_{\R}^k$,  and are closely related (see above) to $\widetilde{CH}^k(X)$. We freely use notation and results from \cite{faselthesis} and \cite{HWXZ}, e.g. when working with twisted coefficients. Recall \cite[Proposition 2.2.27]{Ma20} that the connected components of $X(\R)$ are smooth real manifolds of dimension $d$, and compact if $X$ is projective. General references for real algebraic geometry and cohomology include the books and survey articles of Bochnak-Coste-Roy, Kollar, Mangolte, Silhol, Scheiderer... We also recall that smooth compact topological manifolds have real algebraic models (Nash-Tognoli), but these are highly non-unique, and at least in dimensions $\leq 5$ there is always at least one algebraic model for which the Borel-Haefliger map is surjective \cite[Theorem 11.3.12]{BCR98}.

\medskip

After establishing a slight generalization of Fasel's isomorphism result in top degree to the non-affine case, see Proposition \ref{isok=d}, we concentrate on studying the surjectivity of $c_{\R}$, that is the algebraicity of integral singular cohomology classes. The kernel of $c_{\R}$ will be presumably studied in the forthcoming PhD thesis of Lerbet. Below, we will see that surjectivity of $c_{\R}$ holds in some cases. However, in codimension one, surjectivity already fails for many elliptic curves, see Example \ref{nonconnectedcurves} below. This is in contrast to the complex cycle class map and the target restricted to elements satisfying the Hodge condition (''Lefschetz (1,1)-theorem''). Concerning positive results about the surjectivity of the real cycle class map, we obtain e.g. the following.

\begin{proposition}(see Propositions \ref{K3} and \ref{enriques})
For $X$ a smooth projective real K3 surface, the real cycle class map  $c_{\R}^1:H^1(X,I^1) \to H^1(X(\R),\Z)$ is an epimorphism if and only if the Borel-Haefliger cycle class map $c^1_{BH}$ is an epimorphism. For $X$ a smooth projective real Enriques surface, the real cycle class map  $c_{\R}^1:H^1(X,I^1) \to H^1(X(\R),\Z)$ is always an epimorphism.
\end{proposition}

Hence for real Enriques surfaces, we obtain algebraicity in codimension one with integral coefficients, although a theorem of Mangolte and van Hamel shows that algebraicity with mod 2 coefficients (i.e. surjectivity of $c^1_{BH})$ fails if $X(\R)$ is not orientable.

When looking at surfaces, note that we usually consider cohomology groups for real smooth {\em projective} varieties, in analogy with the Hodge conjecture in the complex case. When it comes to Euler classes for algebraic vector bundles, computations in the affine case are more relevant, of course.

\medskip

We thank Olivier Wittenberg for many very useful explanations and discussions, as well as comments on an earlier draft, and Jean Fasel for useful discussions and for informing us about work in progress by him and Lerbet. 

\section{A few recollections concerning the real cycle class map}\label{recollections}

Let us recall a few fundamental results from the articles above. Essentially by definition, both $I$-cohomology and singular cohomology are concentrated in degrees $0 \leq k \leq d$. We know that $c_{\R}$ is an isomorphism after inverting $2$, i.e. after tensoring with $\Z[1/2]$, and that $c_{\R}$ factors through an isomorphism $H^*(X,I^{\infty})\cong H^*(X(\R),\Z)$, see \cite[Corollary 8.9]{jacobson}. So in some sense the situation here is quite different from the complex case. Indeed, we currently do not know of any general additional condition for algebraicity of cohomology classes in the real case (but see Remark \ref{realconditions} below). What one can try to do in general is to establish upper bounds for the $2$-primary torsion of the kernel and cokernel, see below for a bound of the cokernel following from Jacobson's work. 
\medskip

Without inverting $2$, the real cycle class map is an isomorphism for cellular varieties by \cite[Theorem 5.7]{HWXZ}. The latter article also shows that $c_{\R}$ commutes with pull-backs, push-forwards, ring structures and localization sequences, and all these results extend to twisted coefficients. (These compatibilites may be considered as an integral refinement of \cite[Theorem 11.3.4 and Theorem 11.3.5]{BCR98}.) Finally, we have an isomorphism of Zariski sheaves $I^n \stackrel{\cong}{\to} I^{n+1} \stackrel{\cong}{\to} ...  \stackrel{\cong}{\to} I^{\infty}$ for $n \geq d+1$ by \cite[proof of Corollary 8.11]{jacobson}. We also recall that Balmer and Walter \cite{BW02} establish purity for Witt groups: there is a natural isomorphism $W(X) \stackrel{\cong}{\to} H^0(X,W)$ if $dim(X) \leq 3$. Also, the obvious example $X=Spec(\C)$ shows that already for $d=0$ the map $c_{\R}$ is not always injective, and sometimes not injective for the trivial reason that $X$ doesn't have real points.   

\medskip

We may summarize some of the above facts as follows: the morphisms of Zariski sheaves $I^t \stackrel{<<-1>>}{\to} I^{t+1}$ and the real cycle class map yield a chain of morphisms 
$$H^k(X,W) \to H^k(X,I^k) \to H^k(X,I^{max(k,d+1)}) \stackrel{\cong}{\to} H^k(X,I^{\infty}) \stackrel{\cong}{\to} H^k(X(\R),\Z)$$
where the remaining morphisms become also isomorphisms after applying $\otimes \Z[1/2]$. Finally, the above results show that any element in $H^k(X(\R),\Z)$ which is a multiple of $2^{d-k+1}$ has a lift to $H^k(X,I^k)$. Hence the cokernel of $c_{\R}^k$ is 2-primary torsion bounded by $2^{d-k+1}$. We improve this bound below if $k=d-1$, see Corollary \ref{betterbound}. There is work in progress on improving this existing bound in general by Jean Fasel and Samuel Lerbet. 

\subsection{The real cycle class map and Chow-Witt groups}

Recall (see e.g. \cite[section 2.4]{HW19}) that the map $\widetilde{CH}^k(X,\LC) \to H^k(X,I^k,\LC)$ is always an epimorphism for all $k$, all smooth varieties $X$ and all line bundles $\LC$. Hence the real cycle class map $c_{\R}^k: H^k(X,I^k,\LC) \to H^k(X(\R),\Z(\LC))$ is an epimorphism if and only if the composition $\widetilde{CH}^k(X,\LC) \to H^k(X(\R),\Z(\LC))$ is an epimorphism. Consequently, all statements below on the surjectivity of $c_{\R}^k$ remain true after replacing $H^k(X,I^k,\LC)$ by $\widetilde{CH}^k(X,\LC)$. Following \cite[section 3.D]{HWXZ}, one might argue that rather than this composition one should study the cycle class map from $\widetilde{CH}^k(X,\LC)$ to the fiber product of $I$-cohomology and a certain subgroup of equivariant cohomology. For this, one would have to combine results from this article and from \cite{BW20}, \cite{BW20b}.

\section{General results for $k=d$ and the B\"ar sequence}

In \cite[Corollaire 16.2.2]{faselthesis}, Fasel proves that for affine $X$, the real cycle class map $c_{\R}: H^d(X,I^d,\omega_{X/\R}) \to H^d(X(\R),\Z(\omega_{X/\R}))$, that is for points resp. 0-cycles, is an epimorphism. In fact, in his description the target is a direct sum of copies of $\Z$ indexed by the compact connected components of $X(\R)$, but this is isomorphic to $H^d(X(\R),\Z(\omega_{X/\R}))$ and his map identifies with $c_{\R}$ under this isomorphism, which can be deduced from the following remark.

\begin{remark} For $A=\Z$, $H^0(-,A)$ counts connected components, e.g. by the universal coefficient theorem because $H_0$ does. If $X(\R)$ is compact and oriented, we have Poincar\'e duality, and thus the same statement holds for $H^d$. More generally, Poincar\'e duality yields a statement about $H^d_c$ with coefficients twisted by the orientation line bundle. This essentially implies that $H^d$ with integral coefficients counts compact connected components. With $\Z/2$-coefficients (hence no twists necessary) this is e.g. stated in the introduction of \cite{CTS96}. Both \cite[chapter 3.3]{hatcher} and \cite[chapters B.5 and B.7]{Ma20} contain a collection of general facts on these matters, but don't really discuss the non-oriented case. A useful reference about $H^d$ and connected components in the compact case is \cite[Corollary VI.7.14]{Br93}: in $H^d(X(\R),\Z)$, every oriented component yields a $\Z$ and every non-oriented component yields a $\Z/2$. See also \cite[Theorem V.9.2]{Br97} and \cite[Theorem VI.8.3]{Br93} for Poincar\'e duality with twisted coefficients in the non-oriented case.
\end{remark}

\begin{remark} One might ask if there is something like Poincar\'e duality for $I$-cohomology, looking e.g. at dualities in the motivic stable motivic homotopy category $SH(\R)$ or for generalized (aka MW-) motives $\widetilde{DM}(\R)$ as introduced by Calm\`es-Fasel and studied in detail in \cite{BCDFO}. Indeed, there are very nice duality results for $\widetilde{DM}(\R)$ in \cite{BCDFO}, see e.g. the equalities on pages xvi and xvii. One could try to use these to deduce computations about $H^0$ from those about $H^d$ for smooth projective $X$ with sufficiently simple tangent/orientation  bundles.
\end{remark}

Fasel's proof of surjectivity obviously extends to quasiprojective and in particular to projective $X$. His argument for showing that his map -- which now has become a special case of the more generally defined $c_{\R}$ -- is well-defined uses a reduction to a classical result by Knebusch on curves. The surjectivity is then obvious, and remains obvious for the other description of $c_{\R}$ as well, as any connected component contains a real point, which by definition has an algebraic lift. Fasel also shows \cite[Th\'eor\`eme 16.3.8]{faselthesis} that $c_{\R}$ is injective in this degree if $X$ is affine and oriented, at least if $dim(X) \geq 2$. Note that the injectivity part of this theorem is wrong for $d=0$, just look at  $X=Spec(\C)$. However, it does also hold for affine curves, and for {\em quasi-projective} varieties with a rational point, see Poposition \ref{isok=d} below. 
As Fasel points out on page 151, the orientability assumption for the surjectivity may be removed after checking that the classical results involved generalize well to twists. Moreover, our injectivity result below also holds for twisted coefficients. For a line bundle $\LC$ on a real variety $X$, we will denote the associated local system on $X(\R)$ by $\Z(\LC)$, and refer to \cite[section 2.E.1]{HWXZ} for further details.

\medskip

\begin{question}
One might also ask the following: Is the cokernel of $c_{\R}$ a birational invariant in certain degrees resp. dimensions? For this, one should study and compare long exact sequences with respect to blow ups for the source and target of $c_{\R}$, compare e.g. \cite[Proposition 2.13]{BW20}, keeping in mind that on the topological side blow ups of points are just connected sums with $\R\P^2$, see \cite[section F.3]{Ma20}. We also note that Feld \cite{Fe22} has recently established birational invariance of Chow-Witt groups for zero cycles, and thus for $H^d(-,I^d)$.
\end{question}

\medskip

We now recall the B\"ar sequence, which will be a crucial ingredient for the following computations. In general, we may try to inductively analyse $c_{\R}=c_{\R}^k:H^k(I^k) \to H^k(X(\R),\Z)$ by using the long exact cohomology sequence (''B\"ar sequence'') associated to the short exact sequence of Zariski sheaves $I^{t+1} \to I^t \to I^t/I^{t+1} = \overline{I}^t\cong K^M_t/2$. Starting such an induction requires information about the cohomology groups $H^k(X,I^t)$ for some $t$, usually bigger than $d$ so that the above stabilization holds. It also requires information about $c_{\R}$ on the mod 2-terms, which for $t=k$ corresponds to the well-studied Borel-Haefliger map, and which can sometimes be understood in other bidegrees as well. We recall that the inclusion $I^{t+1} \to I^t$ corresponds to multiplication with $2$ in classical cohomology, and that $c_{\R}$ commutes with the boundary maps by \cite[Proposition 4.14]{HWXZ}. Hence we obtain a commutative ladder of long exact sequences as follows:

\[ 
\xymatrix{ 
  H^{k-1}(X, K^M_t/2) \ar[d] \ar[r] & H^k(X, I^{t+1}) \ar[d] \ar[r] & H^k(X,I^t) \ar[d] \ar[r] & H^k(X,K^M_t/2) \ar[d] \ar[r] & H^{k+1}(X,I^{t+1}) \ar[d]  \\
  H^{k-1}(X(\R),\Z/2) \ar[r] & H^k(X(\R),\Z) \ar[r] & H^k(X(\R),\Z) \ar[r]& H^k(X(\R),\Z/2)  \ar[r] & H^{k+1}(X(\R),\Z)  \\
} 
\]

All this extends to twisted coefficients, and applies to smooth real varieties $X$ of arbitrary dimension $d$. We now apply this sequence to deduce a slight generalization of \cite[Th\'eor\`eme 16.3.8]{faselthesis}.

\begin{proposition}\label{isok=d} Let $X$ be a smooth real (connective) variety of dimension $d$, and $\LC$ a line bundle on $X$. Then the real cycle class map in degree $d$
$$c_{\R}^d:H^d(X,I^d,\LC)\to H^d(X(\R),\Z(\LC))$$
is an epimorphism. If $X$ is affine or if $X$ has a rational point, $c_{\R}^d$ is an isomorphism.
\end{proposition}
\begin{proof}
Surjectivity is obvious as elements in the target are represented by linear combinations of rational points, which lift to generators in $H^d(X,I^d,\LC)$. For the injectivity, we consider the above ladder for $k=t=d$. By the five lemma, the desired monomorphism follows if we have an epimorphism for $H^{d-1}(X,\overline{I}^d)$ and monomorphisms for $H^d(X,I^{d+1})$ and $H^d(X,\overline{I}^d)$. The first monomorphsim follows from Jacobson stability, which extends to twists with line bundles $\LC$ by \cite[p. 299]{HWXZ}. Using the Milnor conjecture, both the epimorphism and the second monomorphism follow from \cite[Theorem 3.2 (d)]{CTS96}, which applies by our assumptions on $X$. 
\end{proof}

\section{Curves}

We now consider the case $d=1$. By Proposition \ref{isok=d}, the real cycle class map $c_{\R}:H^1(X,I^1)\to H^1(X(\R),\Z)$
is always an epimorphism, and often also a monomorphism. Hence let us look at $k=t=0$. At least with trivial twists, the map $W(X) \stackrel{\cong}{\to} H^0(X,W) \to H^0(X(\R),\Z)$ has already been studied by Knebusch \cite{Kn76}. (The identification of his map with $c_{\R}$ follows from the discussion of $\partial_{\mathfrak{p}}$ on page 197 of loc. cit..) Part (iv) of his Theorem 10.4 implies that the cokernel of $W(X) \to H^0(X(\R),\Z)$ is a sum of copies of $\Z/2$ with index set one less than $\pi_0(X(\R))$. 

\begin{example}\label{nonconnectedcurves}
Consider the projective elliptic curve $X$ associated to $y^2=x^3 -x$. We have $X(\R) \approx S^1 \sqcup S^1$, and hence the cokernel is isomorphic to $\Z/2$, represented by $(1,0)$ or $(0,1)$. Hence in cohomological degree zero, the real cycle class map is {\em not} always an epimorphism for smooth curves, in contrast to the classical complex cycle class map (when restricting the target to Hodge classes).
\end{example} 

\begin{remark}\label{realconditions} 
In fact, in loc. cit. Knebusch writes down a certain parity condition for elements being in the image of the Witt group $W(X)$. One might wonder if this result of Knebusch should be considered as a ''real Hodge condition''. However, we observe that it arises by Witt group considerations and purity. Hence it is not clear if similar conditions should be expected in large dimensions.  Another approach to possible algebraicity conditions in arbitrary dimension might come from stable equivariant real realization as studied e.g. in \cite{HO}. More concretely, this means to consider $X(\C)$ as a $G$-space, and to compute the image of the motivic spectrum representing $I$-cohomology under the stable real equivariant realization functor. See also \cite[Remark 2.4 (i)]{BW20} for a corresponding point of view for the other lift of the Borel-Haefliger cycle class map.
\end{remark}
   
\section{Surfaces}

We now consider $d=2$, and we recall that $c_{\R}$ is an isomorphism on $H^2(X,I^2,\LC)$ by Proposition \ref{isok=d} above if $X$ is affine or if $X$ has a rational point. We also recall that $W(X) \cong H^0(X,I^0)$ has been studied in many cases. For instance, \cite[Theorem 2.6, Theorem 3.3]{SvH00} provides a complete computation of $W(X)$ for a real (projective) Enriques surface in terms of the geometry of $X(\R)$, which is well understood (see below). Hence it remains to understand the precise behaviour of $c_{\R}$ on the generators given there.

We will now focus on $H^1(X,I^1)$. For this, we again wish to apply the 
ladder of long exact sequences and the inductive strategy outlined above, starting with $k=1$ and $t=1$. In fact, we have the following general result for arbitrary dimension $d$:

\begin{lemma}\label{keylemma}
Let $X$ be a smooth real variety of dimension $d$ and $\LC$ a line bundle on $X$, and assume that the morphism $c_{BH}^{d-1}: H^{d-1}(X,\overline{I}^{d-1}) \to H^{d-1}(X(\R),\Z/2)$ is surjective. 
Then the real cycle class map $c_{\R}^{d-1}:H^{d-1}(X,I^{d-1},\LC) \to H^{d-1}(X(\R),\Z(\LC))$ is surjective as well.
\end{lemma}
\begin{proof} If $X(\R)=\emptyset$ there is nothing to prove, so we may assume $X$ has a rational point.
By Proposition \ref{isok=d} we then have injectivity of $c^d_{\R}$. Hence by the five lemma, the surjectivity of $c_{\R}^{d-1}:H^{d-1}(X,I^{d-1}) \to H^{d-1}(X(\R),\Z)$ will follow from the surjectivity on $H^{d-1}(X,\overline{I}^{d-1})$ (the classical Borel-Haelfliger map $c_{BH}^{d-1}$) and on $H^{d-1}(X,I^d)$. To investigate the latter, we consider the diagram for $k=d-1$ and $t=d$, which by Jacobson's stability results above yields the surjectivity on $H^{d-1}(X,I^d)$ if we have surjectivity on $H^{d-1}(X,K^M_d/2)$. However, the latter is always surjective by \cite[Theorem 3.2 (d)]{CTS96}, observing once more that $\mathcal{H}^d$ in loc. cit. is the sheaf associated to $K^M_d/2\cong \overline{I}^d$ by the Milnor conjecture. 
\end{proof}

So for surfaces, integral surjectivity follows from surjectivity of $c_{BH}$ on $Pic(X)/2$.  Of course, in general $Pic(X)/2 \to H^1(X(\R),\Z/2)$ is not surjective, as the example $X=\A^2-\{0\}$ shows (see e.g. \cite[Remark 2.60]{HWXZ}). In this case, the following refinement of Lemma \ref{keylemma} is sometimes useful.

\begin{lemma}\label{keylemmarefined}
Let $X$ be a smooth real variety of dimension $d$ and $\LC$ a line bundle on $X$, Then the real cycle class map $c_{\R}^1:H^{d-1}(X,I^{d-1},\LC) \to H^{d-1}(X(\R),\Z(\LC))$ is surjective if and only if the morphism $c_{BH}^{d-1}: H^{d-1}(X,\overline{I}^{d-1}) \to H^{d-1}(X(\R),\Z/2)$ is surjective on all elements lifting to $H^{d-1}(X(\R),\Z(\LC))$.
\end{lemma}
\begin{proof} As in the previous proof, after discarding the case $X(\R)=\emptyset$, we have surjectivity on $H^{d-1}(X,I^d)$ and injectivity for $c^d_{\R}$. If an element $a \in H^{d-1}(X(\R),\Z(\LC))$ is in the image of $c^{d-1}_{\R}$, then obviously its mod $2$ reduction $\overline{a}$ is in the image of $c^{d-1}_{BH}$. Conversely, if $\overline{a}$ is algebraic, then the diagram chase one does for establishing surjectivity in the proof of the five lemma shows in fact that $a$ is of algebraic orign as well.
\end{proof}

In fact, the above proof yields a precise condition for elements in $H^{d-1}(X(\R),\Z(\LC))$ to be algebraic, i.e. in the image of $c^{d-1}_{\R}$. This also leads to a better bound of the 2-primary torsion of the cokernel of the real cycle class map in codimension one.

\begin{corollary}\label{betterbound} 
For $X$, $d$ and $\LC$ as above, an element $a \in H^{d-1}(X(\R),\Z(\LC))$ is algebraic if and only if its reduction $\overline{a} \in H^{d-1}(X(\R),\Z/2)$ is algebraic. In particular, all elements in $H^{d-1}(X(\R),\Z(\LC))$ divisible by $2$ are algebraic, and the cokernel of $c^{d-1}_{\R}$ is 2-torsion.
\end{corollary}

In the following subsections, we will apply these lemmas to several classes of smooth real surfaces. This will lead us to some examples where $c_{\R}^1:H^1(X,I^1) \to H^1(X(\R),\Z)$ is surjective, and to some others where it is not surjective. We recall that the torsion of the cokernel of $c_{\R}^1$ for smooth real surfaces is always bounded by $2^{2-1+1}=4$ thanks to the work of Jacobson.  

\subsection{Geometrically rational surfaces}

For generalities about geometrically rational surfaces, see e.g. \cite[chapters 4.2, 4.4]{Ma20}. For any geometrically rational (projective) surface $X$, we have that $X(\R)$ is empty, a torus or a disjoint union of spheres and connected sums of $\P^2$. This can be deduced from the more precise statements \cite[Proposition 4.4.10, Theorem 4.4.11]{Ma20}, using that by Corollaire F.3.2 in loc. cit. blow ups of points correspond to connected sums with $\R\P^2$.  

The assumption of Lemma \ref{keylemma} for geometrically birational surfaces is satisfied, as the Borel-Haefliger map starting in $Pic(X)/2$ is always an epimorphism for those by \cite[Corollary 3.17, Remark 3.18 (i)]{BW20} and in fact by the more classical references given there.
\begin{proposition}
Let $X$ be a smooth projective geometrically rational surface, and $\LC$ a line bundle on $X$. Then the real cycle class map  $c_{\R}^1:H^1(X,I^1,\LC) \to H^1(X(\R),\Z(\LC))$ is an epimorphism. 
\end{proposition}
\begin{proof}
This follows from Lemma \ref{keylemma} and the references above.
\end{proof}  

\subsection{Enriques surfaces}

By definition (see e.g. \cite{MvH98}), a real Enriques resp. K3 surface is a smooth projective real surface $X$ such that $X_{\C}$ is Enriques resp. K3. Still by definition, K3 surfaces satisfy $\pi_1^{et}(X_{\C})=1$ -- they are 2-fold coverings of Enriques surfaces -- and Enriques surfaces satisfy $H^2(X_{\C},\Oo)=0$, whereas for geometrically rationally surfaces both these groups vanish.  We refer to \cite{DK96} for a complete classification of $X(\R)$ for real Enriques surfaces $X$, with gives rise to 87 different topological types, and which is reproduced in \cite[Theorem 4.5.16]{Ma20}.

A theorem of Mangolte and van Hamel \cite[Theorem 1.1, Theorem 4.4]{MvH98} says that for $X$ Enriques, the map $Pic(X)/2 \to H^1(X(\R),\Z/2)$ is an epimorphism if and only if $X(\R)$ is orientable, which by the above classification is satisfied in some but not all examples.

\begin{proposition}\label{enriques}
Let $X$ be a real Enriques surface. Then the real cycle class map  $c_{\R}^1:H^1(X,I^1) \to H^1(X(\R),\Z)$ is an epimorphism. 
\end{proposition}

Hence we have discovered an important class of real algebraic surfaces for which $c^1_{\R}$ is always surjective, although $c^1_{BH}$ sometimes isn't.

\begin{proof}
If $X(\R)$ is oriented, then this follows from Lemma \ref{keylemma} and the above result of Mangolte and van Hamel.
Hence let us assume that $X(\R)$ is not orientable, so that the Borel-Haefliger cycle class map $c_{BH}^1$ for $CH^1(X)/2$ is not an epimorphism. Of course, it could still be that $c_{\R}^1$ is an epimorphism, because the non-algebraic elements in $H^1(X(\R),\Z/2)$ do not lift to $H^1(X(\R),\Z)$. We now look at this more closely, studying cohomology of non-oriented surfaces and using the finer Lemma \ref{keylemmarefined}. 
Following \cite{DK96}, we write $V_k$ for the non-oriented surface of genus $k$, which may be constructed as the connected sum of $k$ copies of $\R\P^2$. The cohomology ring $H^*(V_k,\Z/2)$ is given by $\Z/2[a_1,...,a_k]/\sim$ with all generators in degree one and relations $a_i a_j=0$ for all $i \neq j$ and $(a_i)^2=1$ for all $i=1,...,k$, see e.g. \cite[Examle 3.8]{hatcher}. We also have $w_1(V_k)=a_1 + ... + a_k$, using that $H^1(\R\P^2,\Z) \stackrel{\cong}{\to} H^1(\R\P^2 -\{pt\},\Z)$ and Mayer-Vietoris. Another standard inductive computation shows that the first integral cohomology group $H^1(V_k,\Z)$ is isomorphic to $\Z^{k-1}$, and we may think of the generators as the circles along which we glue when forming the connected sum. Comparing Mayer-Vietoris sequences, we see that the reduction map $H^1(V_k,\Z) \to H^1(V_k,\Z/2)$ maps the $l$th generator to $a_l + a_{l+1}$. As $(a_l + a_{l+1}) \cdot w_1(V_k)=(a_l)^2 + (a_{l+1})^2 =0$, it follows that all elements in the image of the reduction map are orthogonal (with respect to the cup product) to the first Stiefel-Whitney class. Now consider the decomposition $H^1(X(\R),\Z) \cong H^1(Y,\Z) \oplus \bigoplus_{i \in I} H^1(V_{g_i},\Z)$ with $Y$ oriented (possibly empty) and $I$ finite (possibly empty), which also exists with $\Z/2$-coefficients, and note that the multiplicative structure in cohomology respects this decomposition. We now wish to apply Lemma \ref{keylemmarefined} and the cohomological variant (using Poincar\'e duality) of \cite[Theorem 4.4]{MvH98}. Let $c+b_1 + ... + b_s \in H^1(X(\R),\Z)$ be an element with respect to the above decomposition, and with reduction $\overline{c}+\overline{b_1} + ... + \overline{b_s}$. Now we multiply the latter with $w_1(X(\R))$, which yields $w_1(V_{g_i}) \cdot \overline{b_1} + ... + w_1(V_{g_s}) \cdot \overline{b_s}$. As by assumption all $\overline{b_i}$ come from integral cohomology classes, the above computations for $V_g$ show that all summands $w_1(V_{g_i}) \cdot \overline{b_i}$ are zero, hence so is the sum. Now the above result of Mangolte and van Hamel says that this sum is zero if and only if $\overline{c}+\overline{b_1} + ... + \overline{b_s}$ is in the image of $c^1_{BH}$. Thus we may apply Lemma \ref{keylemmarefined} to conclude.
\end{proof}

For more detailed computations on real Enriques surfaces, see e.g. \cite{MvH98}, \cite{SvH00}, \cite[Example 11.3.17]{BCR98}, \cite[Examples 4.3.1]{CTS96}, \cite{DIK00}, \cite[Examples 4.3]{BW20}, \cite[section 4.5.4]{Ma20}.

\subsection{K3 surfaces}

For $X$ a real smooth projective K3 surface, the possible topological types are again completely classified by Kharlamov \cite{kharlamov76}, yielding a list of 66 topological types. The result is nicely reproduced in \cite[Theorem 4.5.5]{Ma20}. Unlike in the Enriques case, $X(\R)$ is always orientable.

\begin{proposition}\label{K3}
Let $X$ be a real K3 surface. Then the real cycle class map  $c_{\R}^1:H^1(X,I^1) \to H^1(X(\R),\Z)$ is an epimorphism if and only if the Borel-Haefliger cycle class map $c^1_{BH}$ is an isomorphism. 
\end{proposition}
\begin{proof}
As $X(\R)$ is always orientable, the long exact cohomology sequence from $\Z \stackrel{\cdot 2}{\to} \Z \to \Z/2$ yields an epimorphism $H^1(X(\R),\Z) \to H^1(X(\R),\Z/2)$. Hence the claim follows from Lemma \ref{keylemmarefined}. 
\end{proof}  

The map $c^1_{BH}$ for real K3 surfaces has also been studied in great detail, see e.g. Mangolte \cite{Ma97} and \cite[sections 4.5.1 and 4.5.2]{Ma20}. Again, there are precise conditions for which $X$ this map is an epimorphisms and for which it isn't. For examples where we have surjectivity, see e.g. \cite[Examples 4.3.1]{CTS96}, and for examples where surjectivity fails see e.g. \cite[Example 4.5]{BW20}. More details on real K3 surfaces can also be found in \cite[Chapter VIII]{Si89}.

\medskip

\subsection{Abelian surfaces} For real abelian surfaces $X$, the classification can be found in \cite[Theorem 4.5.26]{Ma20}, and again $X(\R)$ is always oriented. We have the following very explicit result.

\begin{proposition}\label{abelian}
Let $X$ be a real abelian surface. Then if the real cycle class map  $c_{\R}^1:H^1(X,I^1) \to H^1(X(\R),\Z)$ is an epimorphism, the space $X(\R)$ is empty or a torus.
\end{proposition}
\begin{proof}
As $X(\R)$ is always orientable, we may argue as in the proof of Proposition \ref{K3}, using once more Lemma \ref{keylemmarefined}. Moreover, the behaviour of $c^1_{BH}$ on real abelian surfaces is well understood thanks to \cite[Proposition 4.5.27]{Ma20}. 
\end{proof}  

Note that we do not say that for all $X$ with $X(\R)$ a torus the map $c^1_{\R}$ really is surjective.

\medskip

Finally, for real bi-elliptic surfaces $X$ excluding the trivial case $X(\R)=\emptyset$, we know by \cite[Theorem 4.5.32]{Ma20} that $c^1_{BH}$ is an epimorphism if and only if $X(\R)$ is a torus and one of the additional conditions of loc. cit. is satisfied. Hence we obtain surjectivity of $c^1_{\R}$ in these cases arguing as in the other cases above. In general, the classification of real bi-elliptic surfaces due to Catanese and Frediani as recalled in \cite[Theorem 4.5.30]{Ma20} contains many non-oriented examples, and also many non-connected examples. We expect that at least in some of the non-connected examples surjectivity of $c^1_{\R}$ fails for reasons similar to Example \ref{nonconnectedcurves}.

\end{document}